\pdfoutput=1
\documentclass[11pt,letterpaper]{article}

\usepackage[margin=1in]{geometry}
\usepackage{amsmath,amssymb,amsthm,mathtools}
\usepackage{booktabs}
\usepackage{tikz}
\usetikzlibrary{calc}
\usepackage[colorlinks=true,citecolor=blue,linkcolor=blue,urlcolor=blue]{hyperref}
\usepackage{cleveref}

\newtheorem{theorem}{Theorem}
\newtheorem{conjecture}[theorem]{Conjecture}
\newtheorem{proposition}[theorem]{Proposition}
\newtheorem{observation}[theorem]{Observation}
\newtheorem{corollary}[theorem]{Corollary}

\newtheorem{remark}[theorem]{Remark}

\DeclareMathOperator{\fix}{fix}
\newcommand{\E}{\mathbb{E}}

\newcommand{\bbZ}{\mathbb{Z}}
\newcommand{\cC}{\mathcal{C}}

\title{Counterexamples to an Extremal Conjecture for Random Cycle-Factors}

\author{
Rishikesh Gajjala\thanks{Supported by Center for Quantum and Topological Systems, NYUAD.}\\
New York University Abu Dhabi}

\date{}
\begin{document}
\maketitle

\begin{abstract}
Christoph, Dragani\'{c}, Gir\~{a}o, Hurley, Michel, and M\"{u}yesser conjectured
that, when \(d\mid n\), the expected number of cycles in a uniformly
random cycle-factor of a directed \(d\)-regular graph on \(n\) vertices
is uniquely maximised by the disjoint union of \(n/d\) copies of the
complete looped digraph \(K_d^\circ\), with value \((n/d)H_d\), in the
extended version of their FOCS 2025 paper.
We disprove this conjecture in the strongest possible range.  For every
\(d\ge 3\) and every multiple \(n=kd\) with \(k\ge 2\), we construct a
directed \(d\)-regular graph on \(n\) vertices whose uniformly random
cycle-factor has expected cycle count strictly larger than \(kH_d\).
We also show that the conjectured extremal picture is correct in degree
\(d=2\), giving a sharp dichotomy between degree two and all higher
degrees.
\end{abstract}
\section{Introduction}

A uniformly random permutation is one of the basic objects of discrete
probability.  It is a classical result that the expected number of
cycles in a uniformly random permutation of \(m\) points is the harmonic
number
$
        H_m:=\sum_{j=1}^m \frac1j
$ \cite{Cauchy1840}. 
Equivalently, if \(K_m^\circ\) denotes the complete looped digraph on
\(m\) vertices, then a uniformly random cycle-factor of \(K_m^\circ\)
is just a uniformly random permutation of \(m\) points, and hence has
expected cycle count \(H_m\).

A natural way to generalise random permutations is to restrict the
allowed images of each point.  Throughout this paper, directed graphs
are finite, may have loops and directed cycles of length two, but have
no parallel edges.  A loop contributes one to both the in-degree and the
out-degree.  A directed graph is \(d\)-regular if every vertex has
in-degree and out-degree exactly \(d\).  A cycle-factor of a directed
graph \(G\) is a permutation \(\sigma\) of \(V(G)\) such that
\(v\to \sigma(v)\) is an edge of \(G\) for every \(v\in V(G)\).  We write
\(\mathcal C(G)\) for the set of cycle-factors of \(G\), and \(c(\sigma)\)
for the number of directed cycles of the permutation \(\sigma\).

This problem is also naturally connected to permanents and perfect
matchings.  Given a directed graph \(G\), form its bipartite double-cover
\(B_G\) with left and right copies of \(V(G)\), and with an edge
\(u_Lv_R\) whenever \(u\to v\) is an edge of \(G\).  Then cycle-factors
of \(G\) are in bijection with perfect matchings of \(B_G\).  Thus
\(|\mathcal C(G)|\) is the permanent of the bipartite adjacency matrix of
\(B_G\).  Classical permanent inequalities give strong evidence that
block constructions should play an extremal role here: the
Brégman--Minc theorem gives sharp upper bounds for such permanents, while
Schrijver's theorem gives sharp lower bounds for the number of perfect
matchings in regular bipartite graphs~\cite{Bregman73,Schrijver98}.
The entropy argument of
Christoph, Draganić, Girão, Hurley, Michel, and Müyesser~\cite{CDGHMM} is closely related to this.

Christoph et al.~\cite{CDGHMM} proved that a uniformly random
cycle-factor of any directed \(d\)-regular graph on \(n\) vertices has
\(O((n\log d)/d)\) cycles in expectation, generalising the classical
fact that a random permutation has logarithmically many cycles.  Their
proof uses entropy to exploit the fact that the upper and lower bounds for permanent for regular bipartite graphs are very close.  They conjectured the following
sharp extremal form of their theorem.

\begin{conjecture}[{\cite[Extended version, Conjecture~4.1]{CDGHMM}}]
\label{conj:CDGHMM}
If \(d\mid n\), then the expected number of cycles in a uniformly random
cycle-factor of a directed \(d\)-regular graph on \(n\) vertices is
uniquely maximised by the disjoint union of \(n/d\) copies of
\(K_d^\circ\).
\end{conjecture}

The value in Conjecture~\ref{conj:CDGHMM} would be \((n/d)H_d\), since
the random cycle-factor splits independently over the \(n/d\) components.
This makes the conjecture very natural: the proposed extremal graph is
the direct analogue of the complete graph for random permutations, and
it is also the clique-type construction suggested by the permanent and
path-partition viewpoints.  The case \(n=d\) is trivial, since a directed
\(d\)-regular graph on \(d\) vertices must be \(K_d^\circ\).  Thus the
first possible range for counterexamples is \(n\ge 2d\).

Cycle-factors are also closely related to path partitions: after
removing one edge from each cycle, one obtains a partition of the
vertices into directed paths.  This is one reason that cycle-factors
appear in work on path partitions and short tours in regular graphs.
Magnant and Martin~\cite{MM09} conjectured that every \(n\)-vertex
\(d\)-regular undirected graph has a path partition with at most
\(n/(d+1)\) paths; this would be tight for the disjoint union of
\(n/(d+1)\) copies of \(K_{d+1}\).  Their conjecture is known for
\(d\le 6\)~\cite{MM09,FF22} and for \(d=\Omega(n)\)~\cite{GL21}.
Related work of Vishnoi~\cite{Vishnoi12} and Feige, Ravi, and
Singh~\cite{FRS14} studies short tours in regular graphs, where
cycle-factors and path partitions provide natural starting points.

In this light, Conjecture~\ref{conj:CDGHMM} proposed a random
cycle-factor analogue of the same clique-type extremal phenomenon. Our
main result shows that this is not the case: although the
disjoint union of complete looped digraphs is the natural candidate, it
does not maximise the expected number of cycles as soon as \(d\ge 3\).


\begin{theorem}\label{thm:main}
For every integer \(d\ge 3\) and every integer \(k\ge 2\), there is a
directed \(d\)-regular graph \(G_{k,d}\) on \(kd\) vertices such that,
for a uniformly random \(\sigma\in \mathcal C(G_{k,d})\),
\[
        \mathbb E c(\sigma) > kH_d .
\]
Consequently, Conjecture~\ref{conj:CDGHMM} fails for every \(d\ge 3\)
and every multiple \(n=kd\ge 2d\).
\end{theorem}

We also show that this failure starts at the first possible degree:
in degree two, the conjectured extremal picture is correct.

\begin{observation}\label{obs:d2}
If \(n\) is even, then the expected number of cycles in a uniformly
random cycle-factor of a directed \(2\)-regular graph on \(n\) vertices
is uniquely maximised, with value
\[
        \frac n2 H_2=\frac{3n}{4},
\]
by the disjoint union of \(n/2\) copies of \(K_2^\circ\).
\end{observation}

Thus the conjecture exhibits a sharp dichotomy: it is true in the only
nontrivial degree below three, but false for every \(d\ge 3\) and every
admissible order \(n\ge 2d\).

Section~\ref{sec:d2} proves Observation~\ref{obs:d2}.  Section~\ref{sec:six}
presents the smallest counterexample, the looped bidirected \(6\)-cycle.
Section~\ref{sec:construction} gives a general counterexample to all
\(d\ge 3\) and proves Theorem~\ref{thm:main}.

We discuss the undirected analogue of Conjecture \ref{conj:CDGHMM} in Appendix \ref{undirected}. We briefly describe the experimental process that led to the counterexamples in this paper in Appendix \ref{appendixAI}. The final constructions and proofs in the main body are self-contained.

\section{\texorpdfstring{The case \(d=2\)}{The case d=2}}\label{sec:d2}

For a directed graph \(G\), let \(\lambda(G)\) denote the number of loops
of \(G\).  For a permutation \(\sigma\), let
\[
        \fix(\sigma):=\bigl|\{v\in V(G):\sigma(v)=v\}\bigr|
\]
be the number of fixed points of \(\sigma\).

\begin{theorem}\label{thm:d2}
Let \(G\) be a directed \(2\)-regular graph on \(n\) vertices, and let
\(\sigma\) be uniformly distributed in \(\mathcal C(G)\).  Then the
following hold.
\begin{enumerate}
\renewcommand{\labelenumi}{\textup{(\arabic{enumi})}}
    \item Every directed edge of \(G\) belongs to \(\sigma\) with probability \(1/2\).
    \item \(\E \fix(\sigma)=\lambda(G)/2\le n/2\).
    \item
    \[
        \E c(\sigma)\le \frac n2+\frac{\lambda(G)}4\le \frac{3n}{4}.
    \]
    \item Equality in the bound \(\E c(\sigma)\le 3n/4\) holds if and
    only if \(n\) is even and
    \[
        G\cong \bigsqcup_{i=1}^{n/2} K_2^\circ .
    \]
\end{enumerate}
\end{theorem}

\begin{proof}
Let $B$ be the bipartite double-cover of $G$.  Its two vertex classes are left and right copies $V_L,V_R$ of $V(G)$, and for every directed edge $u\to v$ of $G$ we put an edge $u_L v_R$ in $B$.  Cycle-factors of $G$ are in natural bijection with perfect matchings of $B$.

Since $G$ is directed $2$-regular, the bipartite graph $B$ is $2$-regular.  Every connected component of a finite $2$-regular graph is a cycle, and every cycle in a bipartite graph is even.  An even cycle has exactly two perfect matchings, and each edge of the cycle lies in exactly one of them.  Hence, under the uniform measure on perfect matchings of $B$, every edge of $B$ is present with probability $1/2$.  Translating back to $G$ proves the first assertion.

A fixed point of $\sigma$ is exactly a loop used by $\sigma$, so the first assertion gives
\[
  \E\fix(\sigma)=\frac{\lambda(G)}2\leq \frac n2.
\]
For any permutation $\pi$ on $n$ points,
\[
  c(\pi)\leq \fix(\pi)+\frac{n-\fix(\pi)}2=\frac{n+\fix(\pi)}2,
\]
because every nontrivial cycle has length at least two.  Taking expectations gives
\[
  \E c(\sigma)\leq \frac{n+\E\fix(\sigma)}2
  =\frac n2+\frac{\lambda(G)}4\leq \frac{3n}{4}.
\]

It remains to identify the equality case.  Suppose that $\E c(\sigma)=3n/4$.  Then all inequalities above are tight.  In particular, $\lambda(G)=n$, so every vertex has a loop.  Since $G$ is directed $2$-regular and has no parallel edges, each vertex has exactly one additional outgoing edge and exactly one additional incoming edge.  These additional edges form the digraph of a fixed-point-free permutation $P$ of $V(G)$.

If $P$ had a cycle of length at least three, then the cycle-factor using this entire $P$-cycle and using loops on all remaining vertices would have positive probability.  For this cycle-factor the pointwise inequality
\[
  c(\pi)\leq \frac{n+\fix(\pi)}2
\]
would be strict, contradicting equality in expectation.  Hence every cycle of $P$ has length two.  Thus $n$ is even, $P$ is a fixed-point-free involution, and $G$ is the disjoint union of $n/2$ copies of $K_2^\circ$.

Conversely, each copy of $K_2^\circ$ has exactly two cycle-factors: the identity, with two cycles, and the transposition, with one cycle.  The expected cycle count on one component is therefore $3/2$, and summing over $n/2$ components gives $3n/4$.
\end{proof}

Since \(H_2=3/2\), Theorem~\ref{thm:d2} proves
Observation~\ref{obs:d2} and verifies Conjecture~\ref{conj:CDGHMM} in
degree two.

\section{A six-vertex counterexample}\label{sec:six}

For $n\geq 4$, define the looped bidirected cycle $G_n^\circ$ on vertex set $\bbZ_n$ by
\[
  i\to i,
  \qquad i\to i+1,
  \qquad i\to i-1 \pmod n.
\]
This is a directed $3$-regular graph as in \Cref{fig:G6}.
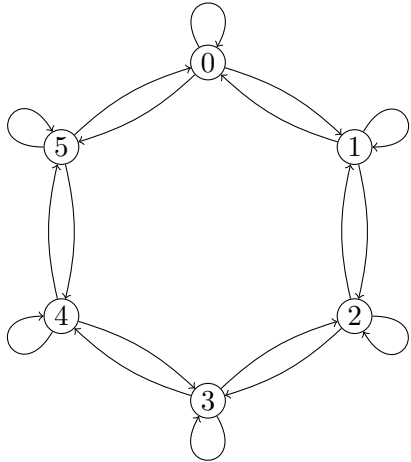
\begin{figure}[t]
\centering
\begin{tikzpicture}[
    scale=1.0,
    every node/.style={circle, draw, inner sep=1.5pt},
    edge/.style={->, bend left=13},
    loopedge/.style={->, looseness=9, min distance=9mm}
]
    \foreach \i in {0,...,5} {
        \node (v\i) at ({90-60*\i}:2.25) {\(\i\)};
    }

    \foreach \i/\j in {0/1,1/2,2/3,3/4,4/5,5/0} {
        \draw[edge] (v\i) to (v\j);
        \draw[edge] (v\j) to (v\i);
    }

    \draw[loopedge] (v0) to[out=120,in=60]   (v0);
    \draw[loopedge] (v1) to[out=60,in=0]     (v1);
    \draw[loopedge] (v2) to[out=0,in=-60]    (v2);
    \draw[loopedge] (v3) to[out=-60,in=-120] (v3);
    \draw[loopedge] (v4) to[out=-120,in=180] (v4);
    \draw[loopedge] (v5) to[out=180,in=120]  (v5);
\end{tikzpicture}
\caption{The looped bidirected \(6\)-cycle \(G_6^\circ\), the smallest
counterexample to Conjecture~\ref{conj:CDGHMM}.}
\label{fig:G6}
\end{figure}
\begin{proposition}\label{prop:cycle-classification}\label{prop:Gn}
Every cycle-factor of $G_n^\circ$ is one of the following:
\begin{enumerate}
  \item the forward Hamilton cycle $i\mapsto i+1$;
  \item the backward Hamilton cycle $i\mapsto i-1$;
  \item a product of adjacent transpositions and fixed points.
\end{enumerate}
Equivalently, apart from the two directed Hamilton cycles, cycle-factors of $G_n^\circ$ are in bijection with matchings of the undirected cycle $C_n$.
\end{proposition}

\begin{proof}
Let $\sigma\in \cC(G_n^\circ)$.  Suppose first that $\sigma(i)=i+1$ and $\sigma(i+1)\neq i$ for some $i$.  Since $\sigma$ is injective and $i+1$ is already used as an image, the value $\sigma(i+1)$ cannot be $i+1$.  Among the three allowed images $i,i+1,i+2$, the only remaining possibility is $i+2$.  Thus $\sigma(i+1)=i+2$.  Repeating the same argument around the cycle gives $\sigma(j)=j+1$ for every $j$, so $\sigma$ is the forward Hamilton cycle.

The same argument with the orientation reversed shows that if $\sigma(i)=i-1$ and $\sigma(i-1)\neq i$ for some $i$, then $\sigma$ is the backward Hamilton cycle.

Therefore, unless $\sigma$ is one of these two Hamilton cycles, every non-loop edge used by $\sigma$ is paired with its reverse.  Hence every vertex which is not a fixed point belongs to a directed $2$-cycle $(i\ i+1)$.  These $2$-cycles are pairwise disjoint, so they are exactly matchings of the underlying undirected cycle $C_n$.
\end{proof}

\begin{corollary}\label{cor:d3-counterexample}
Conjecture~\ref{conj:CDGHMM} is false for $d=3$ and $n=6$.
\end{corollary}

\begin{proof}
By Proposition~\ref{prop:Gn}, the cycle-factors of $G_6^\circ$ consist of the two directed Hamilton cycles together with the cycle-factors arising from matchings of $C_6$.  The numbers of matchings of $C_6$ of sizes $0,1,2,3$ are
$
  1,\ 6,\ 9,\ 2.
$
A matching of size $r$ gives a cycle-factor with $6-r$ cycles.  Hence
\[
  \E c(\sigma)=\frac{2\cdot 1+1\cdot 6+6\cdot 5+9\cdot 4+2\cdot 3}{20}=4.
\]
On the other hand, the expected cycle count of the conjectured extremal graph $K_3^\circ\sqcup K_3^\circ$ is
\[
  2H_3=2\left(1+\frac12+\frac13\right)=\frac{11}{3}<4.
\]
\end{proof}


\section{\texorpdfstring{Counterexamples for \(d\ge 3\)}{Counterexamples for d >= 3}}\label{sec:construction}

One might suspect that the six-vertex example is a small-degree accident,
and that the conjectured family might still be extremal for sufficiently large \(d,n\).
We show that this is not the case.  The construction below is a
degree-\(d\) analogue of the looped bidirected \(6\)-cycle: the six
vertices are replaced by six classes of sizes
\(1,1,d-2,1,1,d-2\), arranged cyclically as in \Cref{fig:Xd}.

\paragraph{Construction.}
Fix \(d\ge 3\).  Define a directed graph \(X_d\) as follows.  Its vertex
set is partitioned into six classes
\[
        A_1,B_1,C_1,A_2,B_2,C_2,
        \qquad |A_i|=|B_i|=1,\qquad |C_i|=d-2.
\]
Inside each class we put the complete looped digraph, and between
consecutive classes in the cyclic order
\[
        A_1,B_1,C_1,A_2,B_2,C_2,A_1
\]
we put all directed edges in both directions.  There are no other edges.
Thus each vertex sees exactly its own class and its two neighbouring
classes as both out-neighbours and in-neighbours, so \(X_d\) is directed
\(d\)-regular.

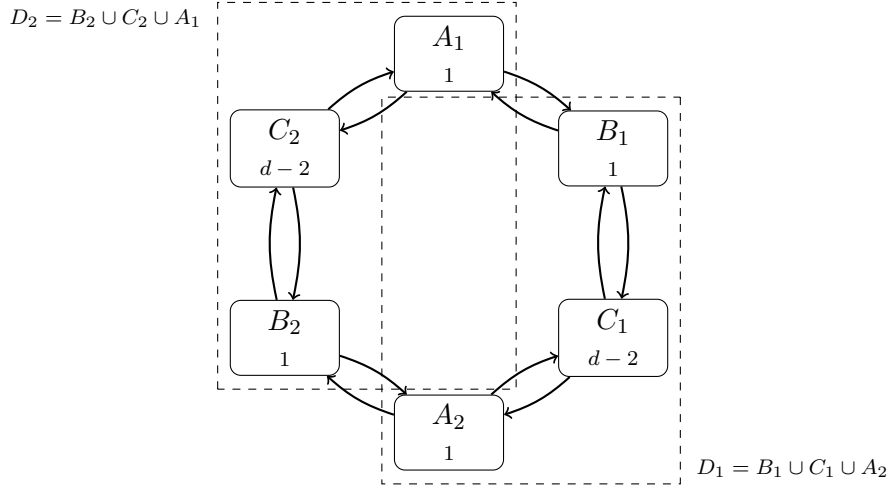
\begin{figure}[t]
\centering
\begin{tikzpicture}[
    scale=1.05,
    box/.style={draw, rounded corners, minimum width=1.45cm, minimum height=.75cm, align=center},
    arr/.style={->, thick, bend left=12}
]
    \node[box] (A1) at (90:2.4)  {\(A_1\)\\{\scriptsize \(1\)}};
    \node[box] (B1) at (30:2.4)  {\(B_1\)\\{\scriptsize \(1\)}};
    \node[box] (C1) at (-30:2.4) {\(C_1\)\\{\scriptsize \(d-2\)}};
    \node[box] (A2) at (-90:2.4) {\(A_2\)\\{\scriptsize \(1\)}};
    \node[box] (B2) at (-150:2.4){\(B_2\)\\{\scriptsize \(1\)}};
    \node[box] (C2) at (150:2.4) {\(C_2\)\\{\scriptsize \(d-2\)}};

    \foreach \X/\Y in {A1/B1,B1/C1,C1/A2,A2/B2,B2/C2,C2/A1} {
        \draw[arr] (\X) to (\Y);
        \draw[arr] (\Y) to (\X);
    }

    \draw[dashed] ($(B1)+(0.85,0.65)$) rectangle ($(A2)+(-0.85,-0.65)$);
    \draw[dashed] ($(B2)+(-0.85,-0.65)$) rectangle ($(A1)+(0.85,0.65)$);

    \node at (4.35,-2.85) {\scriptsize \(D_1=B_1\cup C_1\cup A_2\)};
    \node at (-4.35,2.85) {\scriptsize \(D_2=B_2\cup C_2\cup A_1\)};
\end{tikzpicture}
\caption{The construction \(X_d\).  Each class has all looped directed
edges internally, and consecutive classes in the cyclic order are joined
by all directed edges in both directions.}
\label{fig:Xd}
\end{figure}
A partial permutation is a set of directed edges with no two sharing a tail and no two sharing a head.
\begin{observation}\label{lem:partial-permutation}\label{obs:partial-perm}
Let $F$ be a partial permutation on $K_n^\circ$.  Suppose that $F$ has $r$ prescribed directed edges and $q$ prescribed directed cycle components.  Then
\[
  | \{\pi\in S_n:\pi\supseteq F\}|=(n-r)!,
\qquad
  \sum_{\pi\supseteq F} c(\pi)=(n-r)!\bigl(H_{n-r}+q\bigr).
\]
\end{observation}

\begin{proof}
Every component of $F$ is a directed path or a directed cycle.  Delete the prescribed directed cycle components, remembering that they already contribute $q$ cycles.  Contract each directed path component to a single atom.  Since a path with $\ell$ edges loses $\ell$ vertices under this contraction, and a cycle with $\ell$ edges is deleted together with its $\ell$ vertices, the number of remaining atoms and unused vertices is $n-r$.

A completion of $F$ is exactly a permutation of these $n-r$ objects.  The path contractions do not change the eventual number of cycles, while the deleted directed cycles contribute $q$ cycles outright.  Summing over all permutations of the remaining $n-r$ objects gives
\[
  \sum_{\pi\supseteq F} c(\pi)
  =(n-r)!H_{n-r}+q(n-r)!
  =(n-r)!\bigl(H_{n-r}+q\bigr).
\]
\end{proof}

\begin{theorem}\label{thm:Xd}
Let \(\sigma\) be uniformly distributed in \(\mathcal C(X_d)\).  Then
\[
\E c(\sigma)-2H_d
=
\frac{2(d-2)(3d^3-14d^2+25d-10)}
{d(d-1)(d^4-6d^3+19d^2-30d+20)}
>0.
\]
In particular, \(\E c(\sigma)>2H_d\) for every \(d\ge 3\).
\end{theorem}
\begin{proof}

Set
\[
  D_1:=B_1\cup C_1\cup A_2,
  \qquad
  D_2:=B_2\cup C_2\cup A_1.
\]
Each $D_i$ has size $d$.  Inside $D_1$, the only missing directed edges are
\[
  \bar u_1:B_1\to A_2,
  \qquad
  \bar v_1:A_2\to B_1,
\]
and inside $D_2$, the only missing directed edges are
\[
  \bar u_2:B_2\to A_1,
  \qquad
  \bar v_2:A_1\to B_2.
\]
The four directed edges crossing between \(D_1\) and \(D_2\) are
\[
    u_1:B_1\to A_1,\qquad
    v_1:A_2\to B_2,\qquad
    u_2:B_2\to A_2,\qquad
    v_2:A_1\to B_1 .
\]
We classify a cycle-factor by the subset of these four crossing edges
that it uses.  Since a cycle-factor is a permutation of the vertex set,
the number of chosen edges from \(D_1\) to \(D_2\) must equal the number
of chosen edges from \(D_2\) to \(D_1\).  Equivalently, the number of
chosen edges among \(u_1,v_1\) equals the number of chosen edges among
\(u_2,v_2\).  Thus the possible crossing patterns are
\[
    \varnothing,\quad
    u_1u_2,\quad v_1v_2,\quad
    u_1v_2,\quad v_1u_2,\quad
    u_1v_1u_2v_2 .
\]
Table~\ref{tab:patterns} lists the contribution of each crossing type.
The last column gives the mean number of cycles among cycle-factors of
the indicated type which we now prove.
\begin{table}[t]
\centering
\caption{Cycle-factors of the graph \(X_d\) classified by crossing pattern.}
\label{tab:patterns}
\begin{tabular}{@{}lll@{}}
\toprule
Crossing pattern(s) & Number of cycle-factors & Mean cycle count \\
\midrule
\(\varnothing\)
    & \(N_0^2\)
    & \(2S_0/N_0\) \\[2mm]
\(u_1v_2,\ v_1u_2\)
    & \(2((d-1)!)^2\)
    & \(2H_{d-1}+1\) \\[2mm]
\(u_1v_1u_2v_2\)
    & \(((d-2)!)^2\)
    & \(2H_{d-2}+2\) \\[2mm]
\(u_1u_2,\ v_1v_2\)
    & \(2\bigl((d-2)!(d-2)\bigr)^2\)
    & \(2H_{d-2}-1\) \\
\bottomrule
\end{tabular}
\end{table}

\paragraph{Row 1.} First consider a single block $D_i$ when no crossing edge is used.  Let $N_0$ be the number of permutations of $D_i$ avoiding its two missing opposite edges, and let $S_0$ be the sum of their cycle counts.  By inclusion-exclusion and Observation \ref{lem:partial-permutation},
\begin{align*}
  N_0
  &=d!-2(d-1)!+(d-2)!=(d-2)!(d^2-3d+3),\\
  S_0
  &=d!H_d-2(d-1)!H_{d-1}+(d-2)!\bigl(H_{d-2}+1\bigr).
\end{align*}
The final term in $S_0$ appears because prescribing both missing opposite edges creates a prescribed $2$-cycle. The first row now follows from the definitions of $N_0$ and $S_0$, independently in the two blocks. 

\paragraph{Row 2.} For the boundary $2$-cycle pattern $u_1v_2$, the two crossing edges prescribe the cycle $B_1\leftrightarrow A_1$; after deleting the used row and column in each block, both blocks become free complete looped digraphs on $d-1$ vertices.  Thus this pattern has $((d-1)!)^2$ completions and mean cycle count $2H_{d-1}+1$.  The same argument applies to $v_1u_2$.

\paragraph{Row 3.} For the four-edge pattern $u_1v_1u_2v_2$, the crossing edges prescribe the two boundary $2$-cycles $A_1\leftrightarrow B_1$ and $A_2\leftrightarrow B_2$.  What remains is a free permutation problem on $d-2$ vertices in each block, giving $((d-2)!)^2$ completions and mean cycle count $2H_{d-2}+2$.

\paragraph{Row 4.}
It remains to discuss the path-splicing pattern \(u_1u_2\); the pattern
\(v_1v_2\) is symmetric.  Consider first the block \(D_1\).  The edge
\(u_1\) uses the row of \(B_1\) externally, while the edge \(u_2\) uses
the column of \(A_2\) externally.  Therefore the internal edges chosen
inside \(D_1\) form some directed cycles together with one directed path
whose initial vertex is \(A_2\) and whose terminal vertex is \(B_1\).

Identify \(A_2\) and \(B_1\) to a single distinguished vertex \(*\).
After this identification, the internal configuration becomes a
permutation of \(d-1\) vertices.  The distinguished vertex \(*\) is not
fixed: a fixed point at \(*\) would correspond, before identification,
to the missing directed edge \(A_2\to B_1\).  Conversely, every
permutation of these \(d-1\) vertices in which \(*\) is not fixed opens
uniquely into such an internal configuration in \(D_1\), by cutting the
cycle containing \(*\) at \(*\).

Hence the number of possible internal configurations in \(D_1\) is
\[
        (d-1)!-(d-2)!=(d-2)!(d-2).
\]
Moreover, by Observation~\ref{obs:partial-perm} and inclusion-exclusion,
the mean number of cycles in the contracted permutation, conditioned on
\(*\) not being fixed, is
\[
\frac{(d-1)!H_{d-1}-(d-2)!(H_{d-2}+1)}
     {(d-1)!-(d-2)!}
= H_{d-2}.
\]
Opening the cycle containing \(*\) into a path removes exactly one cycle.
Thus the internal configuration in one block has mean cycle count
\(H_{d-2}-1\).  The same argument applies independently in \(D_2\).
Finally, the two crossing edges \(u_1\) and \(u_2\) splice the two open
paths into one directed cycle.  Therefore the mean cycle count for the
pattern \(u_1u_2\) is
\[
        (H_{d-2}-1)+(H_{d-2}-1)+1
        = 2H_{d-2}-1,
\]
and the number of completions is
\[
        \bigl((d-2)!(d-2)\bigr)^2 .
\]

\paragraph{Computing \(\E c(\sigma)\).}
Let \(N\) be the total number of cycle-factors of \(X_d\), and let
\[
        T:=\sum_{\sigma\in\mathcal C(X_d)} c(\sigma)
\]
be the total cycle-sum.  Summing the corresponding entries of
Table~\ref{tab:patterns} gives
\[
\begin{aligned}
N
&= N_0^2
 +2\bigl((d-2)!(d-2)\bigr)^2
 +2((d-1)!)^2
 +((d-2)!)^2,\\
T
&= 2S_0N_0
 +2(2H_{d-2}-1)\bigl((d-2)!(d-2)\bigr)^2 \\
&\qquad
 +2(2H_{d-1}+1)((d-1)!)^2
 +(2H_{d-2}+2)((d-2)!)^2 .
\end{aligned}
\]
Using
\[
        H_{d-1}=H_d-\frac1d,
        \qquad
        H_{d-2}=H_d-\frac1d-\frac1{d-1},
\]
a direct simplification gives
\[
        N=((d-2)!)^2
        (d^4-6d^3+19d^2-30d+20)
\]
and
\[
        T-2H_dN
        =
        \frac{2((d-2)!)^2(d-2)(3d^3-14d^2+25d-10)}
             {d(d-1)} .
\]
Since \(\E c(\sigma)=T/N\), this proves the displayed formula for
\(\E c(\sigma)-2H_d\).

It remains only to check positivity.  The denominator is positive because
\[
        d^4-6d^3+19d^2-30d+20 = N/((d-2)!)^2>0
\]
and \(d(d-1)>0\).  The polynomial
\[
        f(d):=3d^3-14d^2+25d-10
\]
satisfies \(f(3)=20>0\), while
\[
        f'(d)=9d^2-28d+25>0
\]
for every \(d\ge 3\).  Hence \(f(d)>0\) for all \(d\ge 3\), and therefore
\(\E c(\sigma)>2H_d\).
\end{proof}

For \(d=3\), the construction \(X_3\) is exactly the looped bidirected
\(6\)-cycle \(G_6^\circ\) from Section~\ref{sec:six}, and
Theorem~\ref{thm:Xd} gives \(\E c(\sigma)=4\).

\begin{corollary}\label{cor:all-k}
For every \(d\ge 3\) and every integer \(k\ge 2\), there is a directed
\(d\)-regular graph \(G_{k,d}\) on \(kd\) vertices such that, for a
uniformly random \(\sigma\in\mathcal C(G_{k,d})\),
\[
        \E c(\sigma)>kH_d.
\]
Consequently, Conjecture~\ref{conj:CDGHMM} fails for every \(d\ge 3\)
and every multiple \(n=kd\ge 2d\).
\end{corollary}

\begin{proof}
Let
\[
        G_{k,d}:=
        X_d \sqcup
        \underbrace{K_d^\circ\sqcup\cdots\sqcup K_d^\circ}_{k-2
        \text{ copies}} .
\]
This graph is directed \(d\)-regular and has \(kd\) vertices.
Cycle-factors split over connected components, so
\[
        \mathcal C(G_{k,d})
        =
        \mathcal C(X_d)\times \mathcal C(K_d^\circ)^{\,k-2}.
\]
If
\[
        \sigma=(\sigma_0,\pi_3,\ldots,\pi_k)
\]
is uniformly distributed in this product, then
\[
        c(\sigma)=c(\sigma_0)+\sum_{j=3}^k c(\pi_j).
\]
By linearity of expectation and Theorem~\ref{thm:Xd},
\[
        \E c(\sigma)
        =
        \E c(\sigma_0)+(k-2)H_d
        >
        2H_d+(k-2)H_d
        =
        kH_d.
\]
The graph proposed as extremal in Conjecture~\ref{conj:CDGHMM} is the
disjoint union of \(k\) copies of \(K_d^\circ\), whose expected cycle
count is exactly \(kH_d\).
\end{proof}

\section{Conclusion and open problems}

We disprove Conjecture~\ref{conj:CDGHMM} for every \(d\ge 3\) and every
multiple \(n=kd\ge 2d\).  In degree \(2\), the conjecture is sharp: the
unique extremiser is the disjoint union of \(n/2\) copies of
\(K_2^\circ\).

The main remaining problem is to determine the true extremal value
\[
        M_d(n):=
        \max_G \; \mathbb E_{\sigma\in\mathcal C(G)} c(\sigma),
\]
where the maximum is over all directed \(d\)-regular graphs \(G\) on
\(n\) vertices, and to describe the extremal graphs.

Even the order of the excess over the clique construction is unclear.  By
\Cref{thm:Xd},
\[
        \mathbb E_{\sigma\in\mathcal C(X_d)} c(\sigma)-2H_d
        =
        \frac{6+o(1)}{d^2}
        \qquad \text{as } d\to\infty .
\]
Taking disjoint copies of \(X_d\), and one additional copy of
\(K_d^\circ\) when \(k\) is odd, gives
\[
        M_d(kd)
        \ge
        kH_d+\Omega\!\left(\frac{k}{d^2}\right).
\]
We conjecture that this additive excess is lower order in the natural
asymptotic regime: as \(d\to\infty\),
\[
        M_d(kd)=kH_d(1+o(1)),
\]
uniformly for \(k\ge 2\).  Equivalently, the clique construction should
have the correct leading constant asymptotically, even though it is not
exactly extremal.

It would also be interesting to understand stability.  If a directed
\(d\)-regular graph has expected cycle count close to \(M_d(n)\), must it
be built mostly from some finite family of local structures related to
\(K_d^\circ\) (like \(X_d\))?

Finally, the counterexamples show that the analogy with clique
extremisers in path-partition problems is more subtle than
Conjecture~\ref{conj:CDGHMM} suggested.  It remains open whether the
mechanism in our construction is relevant to extremal questions for path
partitions or short tours, including the conjecture of Magnant and
Martin, or whether those problems are governed by different obstructions.

\bibliographystyle{plainurl}
\bibliography{cycle_factors_refs}

\appendix

\section{Undirected Variants}\label{undirected}
We note that even the undirected variant of Conjecture \ref{conj:CDGHMM} is also not true.
\begin{remark}
In the undirected loopless variant in which a single edge is allowed to count
as a cycle of length two (as in~\cite{CDGHMM}), the analogous clique extremiser
already fails in degree two.  On six vertices, the cycle $C_6$ has exactly
three cycle-factors: the whole $6$-cycle and the two alternating perfect
matchings.  These have respectively $1,3,3$ cycles, and hence the expected
cycle count is
\[
        \frac{1+3+3}{3}=\frac73.
\]
By contrast, every cycle-factor of $2K_3$ consists of the two triangles, so
its expected cycle count is $2$.
\end{remark}
More generally, the switching idea behind our construction has an undirected analogue: one can take three clique blocks, delete one edge in each block, and reconnect the exposed endpoints cyclically.  This gives an obstruction for arbitrarily large degrees.

\begin{remark}
Even under the more restrictive convention that single edges are not allowed
as cycles of length two, the undirected clique construction is not unique at
the putative extremal value. 
Among simple loopless $4$-regular graphs, the
graph $6K_5$ (disjoint union of six $K_5$) has the same expected cycle count for a uniformly random
$2$-factor as the non-clique graph $5K_{2,2,2}$ (disjoint union of five $K_{2,2,2}$).  Indeed, every $2$-factor of
$K_5$ is Hamiltonian, whereas $K_{2,2,2}$ has sixteen Hamiltonian $2$-factors
and four $2$-factors consisting of two triangles, so the expected cycle count is
\[
        \frac{16\cdot 1+4\cdot 2}{20}=\frac65.
\]
Thus both $6K_5$ and $5K_{2,2,2}$ have expected cycle count $6$.
\end{remark}
\newpage
\section{Methodology and AI-assisted discovery}\label{appendixAI}

\paragraph{AI-assisted automated search.}
We used GPT-5.4/5.5 to assist with an automated search over small graphs. The first step was to build an exact verifier. Given a candidate regular graph, the verifier enumerated, or otherwise exactly counted, its cycle-factors, computed the expected number of cycles in a uniformly random cycle-factor, and compared this value with the corresponding clique construction. In particular, the verifier returned the excess over the conjectured value. This allowed candidates to be ranked either by how close they were to being counterexamples, or, when the excess was positive, by the size of the counterexample margin.

We then used PatternBoost~\cite{CEWW24} as a search heuristic. A large number of candidate graphs were generated, and the verifier was used to discard poor candidates and retain promising ones. These retained examples were used as conditioning data for further model-guided generation. The model outputs were again passed through the verifier, bad examples were pruned, and the surviving examples were modified locally. This generate--verify--prune--modify loop was repeated several times. It produced the initial counterexamples in degrees \(3\) and \(4\), which then guided the general construction.

\paragraph{LLM-assisted generalization.}
The computational search was not, by itself, sufficient for the main theorem. After the small examples were found, we looked for structural features that could persist in higher degree. The relevant pattern was that the small examples could be interpreted as clique-like blocks with a small number of deleted internal edges and a small number of compensating edges between blocks. This suggested several possible higher-degree generalisations. These candidate generalisations were again tested by the verifier.

The construction \(X_d\) in Section \ref{sec:construction} emerged from this process. It replaces the six vertices of the looped bidirected \(6\)-cycle by six classes of sizes
\[
1,1,d-2,1,1,d-2
\]
arranged cyclically, with complete looped digraphs inside classes and complete bidirected connections between consecutive classes.

\paragraph{AI-assisted verification.}
The verification of the successful generalisation was also AI-assisted. GPT-5.4/5.5 Pro was used to help organize and check the case analysis for the cycle-factors of \(X_d\). The initial verification by the LLM was substantially more complicated than the proof presented in the paper. After the construction had been identified and computationally tested, we supplied additional guidance about how to structure the analysis: namely, to classify cycle-factors by their crossing pattern between the two \(d\)-vertex blocks \(D_1\) and \(D_2\), and to express the contribution of each crossing pattern using prescribed partial permutations. This led to the shorter proof in Section \ref{sec:construction}, in which the calculation reduces to the four cases displayed in Table \ref{tab:patterns}.

The computational and AI-assisted steps were used to discover, test, and organize the constructions. All final mathematical statements used in the paper are contained in the main text and are verified explicitly.

\end{document}